\documentclass[10pt,twoside,reqno,draft]{amsart}
\usepackage{amsmath}
\usepackage{amssymb}
\usepackage{amsfonts}
\usepackage{amsthm}
\usepackage{graphicx}
\usepackage[initials]{amsrefs}
\usepackage{fancyhdr}
\usepackage{hyperref}

\usepackage[all,cmtip]{xy}

\newtheorem{theorem}{Theorem}[section]
\newtheorem{corollary}[theorem]{Corollary}
\newtheorem{lemma}[theorem]{Lemma}

\theoremstyle{definition}

\newtheorem{remark}[theorem]{Remark}

\begin{document}

\title[Completeness norms]{\large Completeness of constructible norms}
\author[Acosta-Portilla, Sanchez-Nungaray, Ramos-Vázquez and Garrido-Ramírez]{Juan Rafael Acosta-Portilla$^{1}$, Armando Sanchez-Nungaray$^{2}$, Gerardo Ramos-Vázquez$^{3}$ and Lizbeth Yolanda Garrido-Ramírez$^{4}$}
\date{July 2025}
\maketitle

\begin{center}
{\footnotesize
$^{1}$Instituto de Investigaciones y Estudios Superiores Económicos y Sociales, Universidad Veracruzana, México\\
	E-mail: juaacosta@uv.mx\\
\medskip
$^{2}$Facultad de Matemáticas, Universidad Veracruzana, México\\
   E-mail: armsanchez@uv.mx\\
\medskip
$^{3}$Facultad de Matemáticas, Universidad Veracruzana, México\\
    E-mail: ger.ramosv@gmail.com\\
\medskip
$^{4}$Facultad de Economía, Universidad Veracruzana, México\\
	E-mail: ligarrido@uv.mx 
}
\end{center}

\bigskip

{\footnotesize
\noindent
{\bf Abstract.}
In the present work we study for a given Banach space the quotient of the family of equivalent norms modulus the collinearity relationship and prove that under a Banach-Mazur type metric this family is a complete metric space, and we also prove the completeness of the family of constructible norms for the Banach space of bounded Lipschitzian mappings between convex sets of a Banach space.

\noindent
{\bf Key Words and Phrases}:
Renorming, constructible norm, completeness, lipschitzian mapping.

\noindent {\bf 2020 Mathematics Subject Classification}: 
46A99, 46B03, 46B20, 46H20.
}

\bigskip

\section{Introduction}
One of the uses that Renorming Theory has had is to give examples and counterexamples of geometric properties; for a detailed explanation consult \cite{deville1993smoothness, fabian2001functional, godefroy2001renormings, guirao2022renormings}. In the context of the Metric Fixed Point Theory \cite{agarwal2009fixedpointbook, GoebelKirk1990TopicsFixedPoint, kirk2002handbook, takahashi2000book}, one of the great open problems, to date partially proven, is that of the equivalence between the FPP and reflexivity, in 2007 PK Lin \cite{lin2008} proved that there is a renorm of $\ell_1$ with the FPP, while in 2008 Tomás Dominguez \cite{tomas:2009} proved that every reflexive space can be renormed to have the FPP. It is from this that this work intends to respond to question: What is the structure of the family of equivalent norms?

In this article, we expand on the work done by Zizler et al. 1982 \cite{fabianZajicekZizler1982residualityRotundNorms}, Benavides and Phothi 2008 \cite{dominguezPhoti2008porosity}, 2010 \cite{dominguezPhoti2010genericityinsomebanach} on metrics for the space of equivalent norms for a given Banach space and on the work done by Acosta-Portilla and Garrido-Ramírez 2025 \cite{acosta2025constructible} on constructible norms for the space of bounded lipschitzian mappings. Specifically, we investigate the metric structure of the family of equivalent norms by endowing it with a Banach-Mazur type metric, and in particular, we study the metric structure of constructible norms, proving that under the aforementioned metric, they are a complete metric space.

This work is divided into seven sections. The first is this introduction. The second presents the metric studied throughout the text. In the third section, previous results and their relation to the proposed metric are presented. In the fourth section, a proof of the completeness of the space of equivalent norms is given. The fifth section presents the constructible norms and a series of existing results related to them. The sixth section consists of several metric results concerning the constructible norms in relation to the metric formulated in the second section. Finally, the seventh section establishes the completeness of the space of constructible norms.

\section{Distance and collinearity between norms}
Given a Banach space $X$, we will call $\mathcal{N}(X)$ the family of equivalent norms over $X$. Let $\| \cdot \|_1 , \| \cdot \|_2 \in \mathcal{N}(X)$, we will say that $\| \cdot \|_1$ is collinear to $\| \cdot \|_2$ or simply they are collinear if exists $c > 0$ such that

\begin{equation}
\| x \|_2 = c \| x \|_1 
\end{equation}

\noindent for each $x \in X$. On $\mathcal{N}(X)$ we define an equivalence relation $\sim_c$ as $\| \cdot \|_1 \sim_c \| \cdot \|_2$ if $\| \cdot \|_1$ and $\| \cdot \|_2$ are collinear, and we will call

\begin{equation}\label{quotient of collinear norms}
\mathcal{N}'(X) = \mathcal{N}(X) / \sim_c .
\end{equation}

\noindent For each $ \| \cdot \|_1 , \| \cdot \|_2 \in \mathcal{N}(X)$ there exist sharp constants $u \geq l > 0$ such that 

\begin{equation}
l \| x \|_2 \leq \| x \|_1 \leq  u \| x \|_2
\end{equation} 

\noindent for each $x \in X$, under this assumptions we define 

\begin{equation}\label{rho definition}
\varrho (\| \cdot \|_1 , \| \cdot \|_2) = \frac{u}{l}.
\end{equation}

\noindent In the following lemma we will resume the principal properties of function $\varrho$.

\begin{lemma}\label{properties of rho}
Let $X$ be a Banach space, $\| \cdot \|_1 , \| \cdot \|_2 , \| \cdot \|_3 \in \mathcal{N}(X)$ and $r , s > 0$. Then 

\begin{itemize}
\item[1)] $\varrho (\| \cdot \|_1 , \| \cdot \|_2) = \varrho (\| \cdot \|_2 , \| \cdot \|_1)$.
\item[2)] $\varrho (r \| \cdot \|_1 , s \| \cdot \|_2) = \varrho (\| \cdot \|_1 , \| \cdot \|_2)$.
\item[3)] $\varrho(\| \cdot \|_1 , \| \cdot \|_2) \leq \frac{r}{s}$ whenever $s \| x \|_2 \leq  \| x \|_1 \leq r \| x \|_2 $ for each $x \in X$.
\item[4)] $\varrho (\| \cdot \|_1 , \| \cdot \|_2) \leq \varrho (\| \cdot \|_1 , \| \cdot \|_3) \, \varrho (\| \cdot \|_2 , \| \cdot \|_3)$.
\item[5)] $\varrho (\| \cdot \|_1 , \| \cdot \|_2) = 1$ if and only if $\| \cdot \|_1 \sim_c \| \cdot \|_2$.
\end{itemize}
\end{lemma}

Using the results of the previous lemma we conclude that the set of equivalent norms of a Banach space quotient the collinear relation $\mathcal{N}'(X)$ is a metric space if we define 

\begin{equation}\label{definition distance d}
d\left( \overline{\| \cdot \|}_1 , \overline{\| \cdot \|}_2 \right) = \log (\varrho ( \|\cdot \|_1 , \| \cdot \| _2 ))
\end{equation}

\noindent for any $\| \cdot \|_1 \in \overline{\| \cdot \|}_1$ and $\| \cdot \|_2 \in \overline{\| \cdot \|}_2$, we summarize this in the following 

\begin{theorem}\label{d is a distance}
Let $X$ be a Banach space and $d$ as in (\ref{definition distance d}). Then $(\mathcal{N}'(X) ,  d)$ is a metric space.
\end{theorem}

\begin{remark}
The metric $d$ is like the Banach-Mazur distance restricted to the identity operator, see page 15 of \cite{frabian2010banachspacetheroybook}. That is,

\begin{equation}
\begin{split}
d(\| \cdot \|_1 , \| \cdot \|_2) & =  \log \left( \inf  \{ \| T \| \| T^{-1} \|  \, | \, T = I  \} \right) \\ 
                                 & = \log (\| I \| \| I^{-1}\| ).
\end{split}
\end{equation}

\noindent However, the first major difference with the Banach-Mazur distance is that with respect to $d$ the distance $0$ is characterized by collinear norms, unlike $0$ in the Banach-Mazur distance does not imply isometry, see page 276 exercise 5.44 on \cite{frabian2010banachspacetheroybook}. The second and transcendental difference is the completeness of the space, which will be prove later. In addition, it is better to study $d$ as it was constructed in Theorem \ref{d is a distance} since the function $\varrho$ is a measure of change between the Lipschitz constants under renormings.
\end{remark}

We will say that a mapping $T : C \to D$ with $C$ and $D$ subsets of a Banach space $(X, \| \cdot \|)$ is lipschitzian if it has finite Lipschitz value

\begin{equation}
K(T, \| \cdot \|) = \displaystyle{ \sup \left\{ \left. \frac{\| Tx - Ty \|}{\| x - y \|  } \,  \right| \, x , y \in C ,  x \neq y \right\}}
\end{equation}

\noindent The relation of Lipschitz values and $\varrho$ is summarized in the following

\begin{lemma}\label{lemma propertis of rho with K}
Let $T:C \to C$ be a lipschitzian function with $C$ a convex, closed and bounded subset of a Banach space $X$, $\| \cdot \|_1 , \| \cdot \|_2 \in \mathcal{N}(X)$ and $M = \varrho (\| \cdot  \|_1 , \| \cdot \|_2)$. Then
\begin{itemize}
\item[1)]  $K(T , \| \cdot \|_1) = K(T , \| \cdot \|_2)$ whenever $\| \cdot \|_1 \sim_c \| \cdot \|_2$.
\item[2)] $M^{-1} \, K(T , \| \cdot \|_1) \leq K(T , \| \cdot \|_2) \leq M \, K(T , \| \cdot \|_1)$.
\end{itemize}
\end{lemma}

\begin{remark}\label{Remark 1}
By incise $1)$ of Lemma \ref{lemma propertis of rho with K}, every lipschitzian function $T: C \to C$ has a well defined Lipschitz value on each $ \overline{\| \cdot \|} \in \mathcal{N}'(X)$, defined by

\begin{equation}
K\left(T , \overline{\| \cdot \|}\right) := K(T , \| \cdot \|)
\end{equation}
\noindent  for any $\| \cdot \| \in \overline{\| \cdot \|}$, so if we are only interested in studying these constants we can limit ourselves to using the quotient space $\mathcal{N}'(X)$ instead of $\mathcal{N}(X)$.

Incise $2)$ of Lemma \ref{lemma propertis of rho with K} has two important consequences:

\begin{itemize}
\item[i)] Any sequence of operators $(T_n)$ such that $K(T_n , \| \cdot \|) \to 0$ for some $\| \cdot \| \in \mathcal{N}(X)$, also converges to $0$ over each equivalent norm.
\item[ii)] Under every convergent sequence $\left(\overline{\| \cdot \|}_n\right)$ in $\mathcal{N}'(X)$ the respective Lipschitz constants $K(T , \| \cdot \|_n)$ are convergent to $K(T , \| \cdot \|)$ with $ \overline{\| \cdot \|}$ the $d$-lim of $\left(\overline{\| \cdot \|}_n\right)$.
\end{itemize} 
\end{remark}

\section{Comparison of distances for $\mathcal{N}(X)$}

In 1982, Zizler, et al. \cite{fabianZajicekZizler1982residualityRotundNorms} considered the following infinity metric for $\mathcal{N}(X)$. If $X = (X , \| \cdot \|)$ and $B_X = \{ x \in  X \, | \, \| x \| \leq 1 \}$. Then for each $p , q \in \mathcal{N}(X)$ they defined

\begin{equation}
d_{B_X}(p, q) = \sup \{ | p(x ) - q(x) | \, | \, x \in B_X \}.
\end{equation}

\noindent This is clearly the metric of uniform convergence restricted to the unit ball and under this assumption $(\mathcal{N}(x) , d_{B_X})$ is an open subset of the complete metric space formed of all seminorms over $X$ with the same infinity metric. Moreover, it is a Baire space. 

In 2008 Thomas and Supaluk \cite{dominguezPhoti2008porosity} considered the set $\mathcal{E}(X)$ defined as the subset of  $\mathcal{N}(X)$ which members are normalized with respect to any fixed unitary ball $B$, in symbols, 

\begin{equation}\label{normalized norms set definition}
\mathcal{E}(X) = \left\{  p \in \mathcal{N}(X) \, \left| \, \displaystyle{\sup_{x \in B}} \, p(x) = 1 \right. \right\}.
\end{equation}

\noindent and in this work they showed that $(\mathcal{E}(X) , d_{B_X})$ is metrically equivalent to $(\mathcal{E}(X) , d)$, where $d$ is the Banach-Mazur-type metric defined in the previous section. However, they proved that they are not uniformly equivalent. Thus, in other words, the metrics induces the same topology but not the same Cauchy sequences, in consequence they does not necessarily has the same completion.

In 2010 Thomas and Supaluk \cite{dominguezPhoti2010genericityinsomebanach} related the infinity metric to the Hausdorff metric as follows. If $C ,  D \subset X$ are closed and bounded sets, then we define its Hausdorff distance by

\begin{equation}
H(C ,  D) = \max \left\{\displaystyle{\sup_{c \in C}} \, d(c , D) , \displaystyle{\sup_{d' \in D}} \,  d(d' , C) \right\}
\end{equation} 

\noindent where $d(c , D) = \inf \{ \| c -  d' \| \, | \, d '\in D \}$ and $d(d' , C) = \inf \{ \| d' - c \| \, | \, c \in C \}$. For every $p \in \mathcal{N}(X)$ we have its respective $p$-unitary ball $B_p = \{ x \in X \, | \, p(x) \leq 1 \}$. Hence the distance between $p  , q \in \mathcal{N}(X)$ is defined by

\begin{equation}
d_H(p , q) := H(B_p , B_q).
\end{equation}

\noindent Then, in such paper \cite{dominguezPhoti2010genericityinsomebanach}, they proved that over $\mathcal{N}(X)$ the metrics $d_{B_X}$ and $d_H$ are equivalent, but not uniformly equivalent. Now we are going to prove that the spaces $\mathcal{N}'(X)$ and $\mathcal{E}(X)$ are isometrically isomorphic. 

\begin{theorem}
Let $X$ be a Banach space, $\mathcal{E}(X)$ as in (\ref{normalized norms set definition}), $\mathcal{N}'(X)$ as in (\ref{quotient of collinear norms}) and $d$ as in (\ref{definition distance d}). Then  $(\mathcal{N}'(X), d)$ is isometrically isomorphic to $(\mathcal{E}(X), d )$.
\end{theorem}
\begin{proof}
We note that $\mathcal{E}(X)$ is constructed as the image of the function $P$ defined from $\mathcal{N}(X)$ to itself by

\begin{equation}
P \| \cdot \| = \left( \displaystyle{\sup_{x \in B} } \, \| x \| \right)^{-1} \, \| \cdot \|.
\end{equation}

\noindent First, we can prove that $P$ is invariant under collinear norms. That is, if $\| \cdot \|_2 = c\| \cdot \|_1$ for some $c > 0$, then

\begin{equation}
\begin{split}
P \| \cdot \|_2 & = \left( \displaystyle{\sup_{x \in B} } \, \| x \|_2 \right)^{-1} \, \| \cdot \|_2 \\
                & = \left( \displaystyle{\sup_{x \in B} } \, \| x \|_2 \right)^{-1} \, c \| \cdot \|_1 \\
                & = \left( \displaystyle{\sup_{x \in B} } \, c^{-1} \| x \|_2 \right)^{-1} \, \| \cdot \|_1 \\
                & = \left( \displaystyle{\sup_{x \in B} } \, \| x \|_1 \right)^{-1} \, \| \cdot \|_1  \\
                & = P \| \cdot \|_1.
\end{split}
\end{equation}

\noindent Second, in fact $P \| \cdot \|_1 = P \| \cdot \|_2$ implies $\| \cdot \|_1 \sim_c \| \cdot \|$, this follow direct from 

\begin{equation}
\begin{split}
\left( \displaystyle{\sup_{x \in B} } \, \| x \|_2 \right)^{-1} \, \| \cdot \|_2 & = P \| \cdot \|_2 \\
                                                                                   & = P \| \cdot \|_1 \\
                                                                                   & = \left( \displaystyle{\sup_{x \in B_x} } \, \| x \|_1 \right)^{-1} \, \| \cdot \|_1 
\end{split}
\end{equation}

\noindent Hence $\| x \|_2 = c \| x \|_1$ with $c = \left( \displaystyle{\sup_{x \in B} } \, \| x \|_2 \right) \left( \displaystyle{\sup_{x \in B} } \, \| x \|_1 \right)^{-1} $. Thus $\| \cdot \|_1 \sim_c \| \cdot \|_2$ if and only if $P\| \cdot \|_1 = P \| \cdot \|_2$. Hence the following conmutative diagram holds.

\begin{equation}\label{conmutative diagram 1}
\xymatrix{
\mathcal{N}(X) \ar[d]^{P} \ar[r]^{\sim_c} &   \mathcal{N}'(X) \ar[dl]\\
               \mathcal{N}(X) / P \simeq \mathcal{E}(X) \ar[ur] &  
}
\end{equation}

\noindent That is, as sets $\mathcal{N}'(X) \simeq \mathcal{E}(X)$. By (\ref{conmutative diagram 1}) note that the constructed association sends each class to an element of itself $\overline{\| \cdot \|} \mapsto c \| \cdot \| $, so incise 2) of Lemma \ref{properties of rho}  ensures the isometry. 
\end{proof}

\begin{remark}
The reason why it is more convenient to study $(\mathcal{N}'(X) ,  d)$ rather than $(\mathcal{E}(X) ,  d)$ is that the elements of $\mathcal{N}'(X)$ are more flexible in manipulating them, and the definition of $\mathcal{N}'(X)$ does not depend on the choice of a base norm for normalizing the elements.
\end{remark}

\section{Completeness of $\mathcal{N}'(X)$}

In Theorem \ref{N is complete} we will show that $(\mathcal{N}'(X) , d)$ is a complete metric space, however for this we require some definitions and results regarding ultrafilters which can be consulted in \cite{aksoy1990nonstandard} and \cite{heinrich1980ultraproducts}.

A nontrivial ultrafilter $\mathfrak{U}$ over a family of index $I$ is a collection of subsets of $I$ such that

\begin{itemize}
\item[1)] $\emptyset \notin \mathfrak{U}$.
\item[2)] $U \cap V \in \mathfrak{U}$ when $U, V \in \mathfrak{U}$.
\item[3)] If $A \subset  I$ is such that exists $U \in \mathfrak{U}$ with $A \supset U$, then $A \in \mathfrak{U}$. 
\item[4)] For every $A \subset I$, either $A$ or $I \setminus A$ belongs to $\mathfrak{U}$.
\item[5)] There is no finite set $A \subset I$ belonging to $\mathfrak{U}$.
\end{itemize} 

\noindent If $(x_i)_{i\in I}$ is an indexed family on a topological space $(X , \tau_X)$, we will say that $(x_i)$ converges to $x$ with respect to the ultrafilter $\mathfrak{U}$ on $I$ if for every neighborhood $U_x$ of $x$ in $\tau_X$ we have that

\begin{equation}
\{ i \in I \, | \, x_i \in U_x \} \in \mathfrak{U}.
\end{equation}

\noindent in such case we will write

\begin{equation}
x = \displaystyle{\lim_\mathfrak{U}} \, x_i  .
\end{equation}

\noindent It is a well known result that ultrafilters characterize compactness

\begin{theorem}\label{filters convergence over compacts}
Let $K$ be a Hausdorff topological space. Then the following statements are equivalent:

\begin{itemize}
\item[1)] $K$ is compact.
\item[2)] Every indexed family $(x_i)_{i \in I}$ on $K$ is convergent over any ultrafilter $\mathfrak{U}$ over $I$. 
\end{itemize} 
\end{theorem}

\noindent Thus every bounded sequence of real numbers is convergent over any ultrafilter over $\mathbb{N}$. In addition, the limits under ultrafilters are algebraically manipulated in the same way as the usual limits, that is, they are compatible with order relations and vector operations.
 
\begin{theorem}\label{N is complete}
Let $X$ be a Banach space, $\mathcal{N}'(X)$ as in (\ref{quotient of collinear norms}) and $d$ as in (\ref{definition distance d}). Then $(\mathcal{N}'(X) , d)$ is a complete metric space.

\noindent Moreover, if $\mathfrak{U}$ is a nontrivial ultrafilter over $\mathbb{N}$, $\left(\overline{\| \cdot \|}_n\right)$ is a $d$-convergent sequence to $\overline{\| \cdot \|}$  then there exist representatives $\| \cdot \|_n \in \overline{\| \cdot \|}_n$ such that for each $x \in X$

\begin{equation}
\| x \| = \displaystyle{\lim_\mathfrak{U}} \, \| x \|_n.
\end{equation}

\noindent with $\| \cdot \| \in \overline{\| \cdot \|}$. That is, the $d$-limit is equals to a pointwise limit over $\mathfrak{U}$.
\end{theorem}
\begin{proof}
Let $\left(\overline{\| \cdot \|}_n\right)$ be a $d$-Cauchy sequence in $\mathcal{N}'(X)$, $(\varepsilon_n)$ be a sequence of positive real numbers such that $\varepsilon_n \downharpoonright 0$ and $(\varepsilon_n')$ be the sequence $\varepsilon_n' = e^{\varepsilon_n} - 1$. First we will construct the $d$-limit of the sequence. Given $\varepsilon_1 > 0$ there exists $N_1 \in \mathbb{N}$ such that for each $n , m \geq N_1$ we have that

\begin{equation}
d\left(\overline{\| \cdot \|}_m , \overline{\| \cdot \|}_n\right) < \varepsilon_1.
\end{equation}

\noindent We choose a representative $\| \cdot \|_{N_1} \in \overline{\| \cdot \|}_{N_1}$ and for each $n \geq N_1$ we choose the sharpest representative $\| \cdot \|_n \in \overline{\| \cdot \|}_n$ under $\| \cdot \|_{N_1}$. That is, for each $n \geq N_1$ there exists $u_n^1 \geq 1$ such that for each $x \in X$
 
\begin{equation}\label{technical inequality}
\| x \|_n \leq \| x \|_{N_1} \leq u_n^1 \| x \|_n
\end{equation} 
 
\noindent and

\begin{equation}\label{upper bound of u}
u_n^1 = \varrho (\| \cdot \|_n , \| \cdot \|_{N_1}) < e^{\varepsilon_1} = 1 + \varepsilon_1'.
\end{equation}

\noindent Note that

\begin{equation}\label{lower bound of u}
(u_n^1)^{-1} >  (1 + \varepsilon_1')^{-1} > 0.
\end{equation}

\noindent We affirm that for each $x \in X\setminus \{ 0 \}$ the sequence $(\| x \|_n)$ is far from $0$. In fact, we have that 

\begin{equation}\label{sequence away 0}
\begin{split}
\| x \|_n & \geq (u_n^1)^{-1} \| x \|_{N_1} \\
          & > (1 + \varepsilon_1')^{-1} \| x \|_{N_1}\\
          & > 0.
\end{split}
\end{equation}

\noindent From (\ref{technical inequality}) and (\ref{sequence away 0}) it follows that $(\| x \|_n)$ is bounded. Let $\mathfrak{U}$ be a nontrivial ultrafilter over $\mathbb{N}$. Thus for each $x \in X$ we define 

\begin{equation}
\| x \| = \displaystyle{\lim_\mathfrak{U}} \, \| x \|_n.
\end{equation} 

\noindent By (\ref{sequence away 0}) and the linearity of ultra limits, it is clear that $\| \cdot \|$ is a norm over $X$ and by (\ref{upper bound of u}) we have that

\begin{equation}
 1 \leq \displaystyle{\lim_\mathfrak{U}} \, u_n^1 \leq 1 + \varepsilon_1'.
\end{equation}

\noindent Then using (\ref{technical inequality})

\begin{equation}
\begin{split}
\| x \| & = \displaystyle{\lim_\mathfrak{U}} \, \| x \|_n \\
 & \leq  \| x \|_{N_1} \\
                                               &  \leq \displaystyle{\lim_\mathfrak{U}} \,( u_n^1 \, \| x \|_n ) \\
                                               &  =  \left (\displaystyle{\lim_\mathfrak{U}} \, u_n^1 \right)\, \left( \displaystyle{\lim_\mathfrak{U}} \, \| x \|_n \right) \\
                                               & \leq \left(1 + \varepsilon_1' \right) \, \| x \|
\end{split}
\end{equation}

\noindent Hence $\| \cdot \| \in \mathcal{N}(X)$ and by incise $3)$ of Lemma \ref{properties of rho} we have that 

\begin{equation}\label{distance for N1}
\varrho( \| \cdot \| , \| \cdot \|_{N_1} ) \leq 1 + \varepsilon_1'.
\end{equation}

Now we going to construct a subsequence of $\left(\overline{\| \cdot \|}_n\right)$ such that it $d$-converge to $\overline{\| \cdot \|}$. We will proceed inductively, with the base step the norm $\| \cdot \|_{N_1}$ which satisfies (\ref{distance for N1}). Let $k \in \mathbb{N}$ and we assume $\|\cdot \|_{N_s}$ chosen for $s <k$. Since $\varepsilon_k > 0$, then there exists $N_k > N_{k-1}$ such that for each $n , m  \geq N_k$ 

\begin{equation}
d \left(\overline{\| \cdot \|}_n , \overline{\| \cdot \|}_m\right) < \varepsilon_k
\end{equation} 

\noindent or equivalently 

\begin{equation}\label{technical distance rho k}
\varrho(\| \cdot \|_n , \| \cdot \|_m) < 1 + \varepsilon_k'.
\end{equation}

\noindent Thus, for each $n \geq N_k $ there exist sharp $u_n^k \geq l_n^k > 0$ such that for each $x \in X$

\begin{equation}\label{technical inequality 2}
l_n^k \| x \|_n \leq \| x \|_{N_k} \leq u_n^k \| x \|_n
\end{equation}

\noindent By (\ref{technical inequality}), (\ref{upper bound of u}) and (\ref{lower bound of u}), we have that for each $n \geq N_k$ and $x \in X$

\begin{equation}
\begin{split}
(1 + \varepsilon_1')^{-1} \| x \|_n & \leq (1 + \varepsilon_1')^{-1} \| x\|_{N_1} \\
                                       & \leq \| x \|_{N_k} \\
                                       & \leq \| x \|_{N_1} \\
                                       & \leq (1 + \varepsilon_1') \| x \|_n                                     
\end{split}
\end{equation}

\noindent So from the fact that $l_n^k$ and $u_n^k$ are sharp and incise $3)$ of Lemma \ref{properties of rho}, it follows that for each $n \geq N_k$

\begin{equation}
0 < (1 + \varepsilon_1')^{-1} \leq l_n^k \leq u_n^k \leq 1 + \varepsilon_1'.
\end{equation}

\noindent Then

\begin{equation}
0 < (1 + \varepsilon_1')^{-1} \leq \displaystyle{\lim_{n , \mathfrak{U}}} \, l_n^k \leq \displaystyle{\lim_{n , \mathfrak{U}}} \, u_n^k \leq 1 + \varepsilon_1'
\end{equation}

\noindent That is, the ultralimits exist. Hence by (\ref{technical inequality 2})

\begin{equation}
\begin{split}
\left( \displaystyle{\lim_{n , \mathfrak{U}}} \, l_n^k \right) \| x \| & = \displaystyle{\lim_{n , \mathfrak{U}}}  \left( l_n^k  \| x \|_n \right)  \\
         & \leq \| x \|_{N_k} \\
         & \leq  \displaystyle{\lim_{n , \mathfrak{U}}} \left( u_n^k \| x \|_n \right) \\
         & = \left( \displaystyle{\lim_{n , \mathfrak{U}}} \, u_n^k \right) \| x \|.
\end{split}
\end{equation}

\noindent From (\ref{technical distance rho k}) it follows that $\displaystyle{\frac{u_n^k}{l_n^k}} < 1 + \varepsilon_k'$ for each $n  \geq N_k$, therefore

\begin{equation}
\displaystyle{\frac{ \displaystyle{\lim_{n , \mathfrak{U}}} \, u_n^k }{\displaystyle{\lim_{n , \mathfrak{U}}} \, l_n^k}} \leq 1 + \varepsilon_k'
\end{equation}

\noindent In other words, 

\begin{equation}
d ( \| \cdot \|, \| \cdot \|_{N_k}) \leq \varepsilon_k
\end{equation}

\noindent This is how we have constructed the subsequence $\left(\overline{\| \cdot \|}_{N_k}\right)$ of $\left(\overline{\| \cdot \|}_n\right)$ that $d$-converges to $\overline{\| \cdot \|}$ and since $\left(\overline{\| \cdot \|}_n\right)$ is Cauchy we conclude that it converges to $\overline{\| \cdot \|}$. 
\end{proof}

\begin{corollary}\label{Corollary form of limit as U pointwise limit}
Let $X$ be a Banach space, $\varepsilon> 0$, $\mathfrak{U}$ be a nontrivial ultrafilter over $\mathbb{N}$, $\left(\overline{\| \cdot \|_n}\right)$ be a $d$-convergent sequence in $\mathcal{N}'(X)$ with $d$-limit $\overline{\| \cdot \|} \in \mathcal{N}'(X)$, $N \in \mathbb{N}$ and for each $n \geq N$ a representative $\| \cdot \|_n \in \overline{\| \cdot \|_n}$ and $u_n \geq 1$  such that in a sharp way for each $x \in X$

\begin{equation}
\| x \|_n \leq \| x\|_N \leq u_n \| x \|_n
\end{equation}

\noindent with $d(\| \cdot \|_n , \| \cdot \|_N) = \log u_n < \varepsilon$. Then the function defined for each $x \in X$ by

\begin{equation}
\| x \| = \displaystyle{\lim_\mathfrak{U}}\, \| x \|_n 
\end{equation} 

\noindent is a representative of the $d$-limit $\overline{\| \cdot \|}$.
\end{corollary}

\begin{lemma}\label{Lemma infinity norm as limit of infinity norms}
Let $X$ be a Banach space, $C$ a nonempty convex, closed and bounded subset of $X$, $\mathfrak{U}$ be a nontrivial ultrafilter over $\mathbb{N}$, $(\| \cdot \|_n)$ be a sequence in $\mathcal{N}(X)$ such that, there exists $\| x \| := \displaystyle{\lim_\mathfrak{U}} \, \| x \|_n$ for each $x \in X$. Then for each bounded $T:C \to X$ 

\begin{equation}
\| T \|_\infty = \displaystyle{\lim_\mathfrak{U}} \, \| T \|_{n , \infty}
\end{equation}

\noindent is fulfilled.
\end{lemma}

\begin{proof}
Let $\varepsilon > 0$, $T \in BLip(C,X)$ and we call 

\begin{equation}
a= \displaystyle{\lim_\mathfrak{U}} \, \| T \|_{n, \infty} =\displaystyle{\lim_{n , \mathfrak{U}} \, \sup_{x\in C}} \, \|Tx \|_n.
\end{equation}

\noindent Then there exists $I \in \mathfrak{U}$ such that for each $n \in I$ we have that

\begin{equation}
a - \varepsilon < \| T \|_{n , \infty} = \displaystyle{\sup_{x \in C}} \, \| T x \|_n.
\end{equation}

\noindent Thus for every $n \in I$ there exist $x_n \in C$ such that

\begin{equation}
a-\varepsilon < \| Tx_n \|_n.
\end{equation}

\noindent Hence

\begin{equation}
\begin{split}
a- \varepsilon & \leq \displaystyle{\lim_{n \in I , \mathfrak{U}}} \, \| Tx_n \|_n \\
               & = \displaystyle{\lim_{n, \mathfrak{U}}} \, \| Tx_n \|_n \\
               & = \| Tx_n \| \\
               & \leq \displaystyle{\sup_{x \in C}} \| Tx \|.
\end{split} 
\end{equation}

\noindent Then

\begin{equation}
\displaystyle{\lim_\mathfrak{U} }\, \| T \|_{n, \infty} \leq  \| T\|_\infty.
\end{equation}

On the other hand. For each $x \in C$ we have that 

\begin{equation}
\| Tx \|_n \leq \displaystyle{\sup_{x \in C}} \, \| Tx \|_n
\end{equation}

\noindent Then for each $x \in X$ 

\begin{equation}
\displaystyle{\lim_\mathfrak{U}  } \, \| T x \|_n \leq \displaystyle{\lim_\mathfrak{U} \, \sup_{x \in C}} \, \| Tx \|_n
\end{equation}

\noindent Finally, since the right side of the last inequality does not depend on $x \in C$, then

\begin{equation}
\begin{split}
\|T \|_\infty & = \displaystyle{\sup_{x \in C} \, \lim_\mathfrak{U}} \, \|Tx \|_n \\
              & \leq \displaystyle{\lim_\mathfrak{U} \, \sup_{x \in C}} \, \| T x \|_n \\
              & = \displaystyle{\lim_\mathfrak{U} } \, \| T \|_{n , \infty}.
\end{split}
\end{equation}
\end{proof}

\section{Constructible norms for $BLip(C, X)$}

In this section we will introduce the family of constructible norms for a collection of bounded lipschitzian mappings on a Banach space and some known results about them. 

Let $(X , \| \cdot \|)$ be a Banach space and $C$ be a nonempty subset of $X$. Then we call

\begin{equation}
BLip(C,X) = \{ T: C \to X \, | \, K(T, \| \cdot \| ) , \| T \|_\infty < \infty \}
\end{equation}

\noindent the family of bounded lipschitzian mappings from $C$ to $X$. It is not hard to check that $BLip(C,X)$ is independent of the choice of a base norm as long as the norm is equivalent. It is a well known result that $BLip(C,X)$ is a Banach space if we endowed it with the norm 

\begin{equation}
\| T \|_b = \| T \|_\infty + K(T , \| \cdot \|),
\end{equation}

\noindent see for example the classic \cite{weaver2018lipschitz} or \cite{acosta2021intersection}. We will be considering $\mathcal{N}(BLip(C,X))$ as the family of equivalent norms to $ \| \cdot \|_b$. The following definitions can be found in \cite{acosta2025constructible} and it is a family of norms that generalize to norm $\| \cdot \|_b$. If $\| \cdot \| \in \mathcal{N}(X) $ and $\theta \in \mathcal{N}(\mathbb{R}^2)$ we define for each $T \in BLip(C,X)$

\begin{equation}
\theta \| T \| := \theta (\| T \|_\infty , K (T , \| \cdot \|)),
\end{equation}

\noindent we say that $\theta \| \cdot \|$ is a constructible norm for $BLip(C,X)$ and it is constructed by $\| \cdot \|$ and $\theta$. We denote the family of constructible norms for $BLip(C,X)$ by 

\begin{equation}
\mathcal{N}(\mathbb{R}^2) \mathcal{N}(X) \subset \mathcal{N}(BLip(C,X))
\end{equation}

\noindent and the family of constructible norms modulus the collinearity relationship by

\begin{equation}
(\mathcal{N}(\mathbb{R}^2) \mathcal{N}(X) )'\subset \mathcal{N}'(BLip(C,X)).
\end{equation}

\noindent Note that constructible norms $\rho \| \cdot \|$ only take into account positive arguments when evaluating the norm $\rho(\, \cdot  ,  \cdot ) $ of $\mathbb{R}^2$, because $\| \cdot \|_\infty$ and $K (\, \cdot \, , \| \cdot \|)$ are non negative. Therefore, from now on we will only consider the norms $\rho(\, \cdot , \cdot )$ defined in the first quadrant of $\mathbb{R}^2$ unless otherwise mentioned. We will say that two operators $S, T \in BLip(C , X)$ are $\| \cdot \|$-indistinguishable if $\| S \|_\infty = \| T \|_\infty$ and $K(S , \| \cdot \|) = K(T , \| \cdot \|)$. To each norm $\eta(\cdot ) \in \mathcal{N}(BLip(C , X))$ we can associate a norm over $X$ defined by

\begin{equation}\label{definition of projection of BLip norms to X norms}
\phi (\eta(x)) = \eta (f_x)
\end{equation}

\noindent for each $x \in X$ where $f_x:C \to X$ is the constant function $x$. We will cal $\phi$ the projection of $\mathcal{N}(BLip(C,X))$ in $\mathcal{N}(X)$. Since $\phi$ is an evaluation, then $\phi (c \rho(\cdot ))(x) = c \rho(f_x)$, thus $\phi$ is well defined in the quotient space of collinear norms, that is, $\phi : \mathcal{N}'(BLip(C,X)) \to \mathcal{N}'(X)$. It follows direct from definition that 

\begin{lemma}\label{Lemma phi is nonexpansive}
Let $X$ be a Banach space, $C$ a nonempty convex, closed and bounded subset of $X$ and $\phi$ defined as in (\ref{definition of projection of BLip norms to X norms}). Then 

\begin{equation}
\phi: \mathcal{N}'(BLip(C,X)) \to \mathcal{N}'(X)
\end{equation}

\noindent is $d$ to $d$ nonexpansive.
\end{lemma}

If $x , y \in X$ then we call $[x , y] = \{\lambda x + (1 - \lambda)y \, | \, 0 \leq \lambda \leq 1 \}$ the convex hull of $x$ and $y$. In \cite{acosta2025constructible} was prove the following five results that we will use in this work:

\begin{lemma}\label{retract lemma}
Let $(X, \| \cdot \|)$ be a normed space, $x , y \in X$ and $Z \subset Y$, $N \supset [x, y]$ be metric spaces then:
\begin{itemize}
\item[(1)] Every bounded Lipschitzian function $f:Z \to [x , y]$ can be extended to a bounded Lipschitzian function $F: Y \to [x, y]$ with the same $\| \cdot \|$-Lipschitz constant and $\| \cdot \|$-infinity norm.

\item[(2)] Every bounded Lipschitzian function $f:[x, y] \to Z$ can be extended to a bounded Lipschitzian function $F: N \to Z$ with the same $\| \cdot \|$-Lipschitz constant and $\| \cdot \|$-infinity norm.
\end{itemize}
\end{lemma}

\begin{lemma}\label{bound2}
Let $(X, \| \cdot \|)$ be a normed space, $C$ a subset of $X$ with at least 
two elements, $\varepsilon \geq 0$ and 

\begin{equation*}
s = \sup \left\{\frac{2 \varepsilon}{\|x-y \|}: x, y \in C, x \neq y \right\}.
\end{equation*}

\noindent Then for each $ 0 \leq r <   s$ there exists $T \in BLip(C , X)$ such that 
$\| T\|_\infty = \varepsilon$ and $K(T, \| \cdot \|) = r$.

\noindent Additionally, if $s < \infty$. Then the previous statement is valid for each 
$0 \leq r \leq s$.
\end{lemma}

\begin{theorem}\label{EquivSeparation}
Let $X$ be a normed space and $\| \cdot \|_1 , \| \cdot \|_2 \in \mathcal{N}(X)$ then the following statements are equivalent:
\begin{itemize}
 \item[(1)] $\|\cdot \|_1 \sim_c \| \cdot \|_2$.
 \item[(2)] There exist $\rho_1 , \rho_2 \in \mathcal{N}(\mathbb{R}^2)$ such that $\rho_1 \| \cdot \|_1 = \rho_2 \| \cdot \|_2$.
 \item[(3)] For each $\rho_1 \in \mathcal{N}(\mathbb{R}^2)$ exists $\rho_2 \in \mathcal{N}(\mathbb{R}^2)$ such that $\rho_2 \| \cdot \|_2 =\rho_1 \| \cdot \|_1$.
\end{itemize}
\end{theorem} 

\begin{lemma}\label{TecCol}
Let $X$ be a normed space, $C$ a nonempty subset of $X$, $T \in BLip(C,X)$, $\rho \in \mathcal{N}(\mathbb{R}^2) $ and $\| \cdot \| \in \mathcal{N}(X)$, Then:

\begin{itemize}
\item[(1)] $\phi ( \rho \| \cdot \|) \sim_c \| \cdot \|$.
\item[(2)] $\phi (\rho \| T \| )_\infty = \rho(1, 0)\| T \|_\infty$.
\item[(3)] $K(T , \phi(\rho \| \cdot \|)) = K(T, \| \cdot \|)$. 
\end{itemize}
\end{lemma}

\begin{theorem}\label{Theorem construct Intrin Charac}
Let $X$ be a normed space, $C$ a nonempty subset of $X$ and $\eta ( \cdot ) \in \mathcal{N}(BLip(C, X))$, Then the following statements are equivalent:
\begin{itemize}
\item[(1)] $\eta(\cdot ) \in \mathcal{N}(\mathbb{R}^2) \mathcal{N}(X)$.
\item[(2)] For each $T, S \in BLip(C, X)$ that are $\phi (\eta(\cdot))$-indistinguishable is satisfied $ \eta( T ) = \eta ( S )$.
\item[(3)] $\eta (\cdot ) \in \mathcal{N}(\mathbb{R}^2) \phi (\eta(\cdot ))$.
\end{itemize}
\end{theorem}

\section{A first metrical approach to constructible norms}

In this section, we demonstrate a series of technical results that will serve to demonstrate the completeness of the space of constructible norms and that refer to the continuity of the values adopted by the infinite norm and the Lipschitz constant, the structure of equivalence classes of constructible norms, and as a main result we present the relationship between the distance between constructible norms and the distances of their constitutive norms.

\begin{lemma}\label{values of infinite and Lipschitz}
Let $(X , \| \cdot \|)$ be a Banach space and $C\subset X$ a closed, bounded and convex set with at least two elements. Then for each $\varepsilon > 0$ and $s \geq 0$ there exists $T \in BLip(C,X)$ such that $\| T \|_\infty = \varepsilon$ and $K(T , \| \cdot \|) = s$.
\end{lemma}

\begin{proof}
Since $C$ is a nontrivial convex, then $\inf \{ \| x - y \| \, | \, x , y \in C , x \neq y \} = 0$. Thus for each $\varepsilon > 0$ we have that $\sup \left\{\frac{2 \varepsilon}{\|x-y \|}: x, y \in C, x \neq y \right\} = \infty$. Then by Lemma \ref{bound2} the conclusion follows. 
\end{proof}

\begin{lemma}\label{Lemma form of constructible norms}
Let $X$ be a Banach space, $C$ a convex, closed and bounded subset of $X$ with at least two elements and $\overline{\rho \| \cdot \|} \in (\mathcal{N}(X) \mathcal{N}(\mathbb{R}^2))'$ be the equivalence class of a constructible norm. Then for each $\eta \in \overline{\rho \| \cdot \|}$ and $\| \cdot \|_c \sim_c \| \cdot \|$ there exist $d,e > 0$ such that for each $T \in BLip(C,X)$

\begin{equation}\label{form of constructible norms}
\eta(T ) = d  \, \rho(e^{-1} \,  \| T \|_{c,\infty} , K(T, \| \cdot \|_c)).
\end{equation}

\noindent Moreover (\ref{form of constructible norms}) is the general form of the elements in $\overline{\rho \| \cdot \|}$. That is, if $\| \cdot \|' \in \mathcal{N}(X)$ and $\rho' \in \mathcal{N}(\mathbb{R}^2)$ exist such that $\rho' \| T \|' = \rho \|T\|$ for each $T \in BLip(C,X)$. Then $\| \cdot \|' = a \| \cdot \|$ for some $a > 0$ and $\rho'( \cdot , \cdot ) = \rho(a^{-1} \, \cdot , \cdot )$.
\end{lemma}

\begin{proof}
Since the elements in $\overline{\rho \| \cdot \|}$ are scalar multiple of the norm $\rho \| \cdot \|$, then exists $d > 0$ such that $ \eta(\cdot ) = d ( \rho \| \cdot \|) = (d \rho)\| \cdot \|$. By hypothesis we have that $\| \cdot \|_c = e \| \cdot \|$ for some $e > 0$. Then for each $T \in BLip(C,X)$

\begin{equation}
\begin{split}
\eta(T) & = (d \rho) \| T \| \\
        & = d \rho ( \| T \|_\infty , K( T , \| \cdot \|)) \\
        & = d \rho (e^{-1} e \| T \|_\infty , K(T, \| \cdot \|_c)) \\
        & = d \rho (e^{-1}\| T \|_{c,\infty} , K(T, \| \cdot \|_c)).
\end{split}
\end{equation}

Now we will prove that (\ref{form of constructible norms}) is unique in its expression. If $\rho' \| T \|' = \rho \| T \|$ for each $T \in BLip(C,X)$, then by Theorem \ref{EquivSeparation} $\| \cdot \|' = a \| \cdot \|$ for some $a > 0$ and we have that 

\begin{equation}\label{equality of rho for constructible norms}
\begin{split}
\rho (\| T \|_\infty , K(T , \| \cdot \|)) & = \rho \| T \| \\
                                                    & = \rho'\| T \|' \\
                                                    & = \rho'(\| T \|'_\infty , K(T, \| \cdot \|')) \\
                                                    & = \rho'(a \| T \|_\infty , K(T , \| \cdot \|)).
\end{split}
\end{equation}

\noindent By Lemma \ref{values of infinite and Lipschitz} it follows that (\ref{equality of rho for constructible norms}) is valid for each combination of positive values on the arguments $\| T \|_\infty$ and $K(T , \| \cdot \|)$. Hence as functions $ \rho' ( \, \cdot , \cdot ) = \rho ( a^{-1} \, \cdot , \cdot )$. 
\end{proof}

\begin{lemma}\label{Lemma technical indistinguible under two norms}
Let $X$ be a Banach space, $C$ a convex, closed and bounded subset of $X$ with at least two different elements and $\| \cdot \|_1 , \| \cdot \|'_2 \in \mathcal{N}(X)$. Then exists $\| \cdot \|_2 \sim_c \| \cdot \|'_2$ such that for each $r > 0 $ and $s \geq 0$ there exists $T \in BLip(C, X)$ with

\begin{equation}
\| T \|_{1, \infty} = \| T \|_{2 , \infty} = r
\end{equation}

\noindent and

\begin{equation}
K(T , \| \cdot \|_1) = K(T , \| \cdot \|_2) = s
\end{equation}
\end{lemma}
\begin{proof}
Let  $x , y \in C$ with $x \neq y$ and $w' = x - y$. Then there exists $c > 0$ such that $c \|w' \|'_2 = \| w' \|_1$. We call $\| \cdot \|_2 = c \| \cdot \|'_2$ and we choose $w = \alpha' w'$ with $\alpha' > 0$ such that 

\begin{equation}\label{technical same infinity norm}
\|w \|_1 = \|w \|_2 = r,
\end{equation}
 
\noindent note that 

\begin{equation}\label{technical same norm on ray}
\|\alpha w \|_2 = \| \alpha w \|_1
\end{equation}

\noindent for each $\alpha \in \mathbb{F} = \mathbb{R} \vee \mathbb{C}$. Since $C$ is a nontrivial convex, then we can chose $x'  \in [x , y]\subset C$ with $x' \neq y$ and $0 \leq \delta \leq 1$ such that 

\begin{equation}
\displaystyle{\frac{\delta \| w \|_2}{ \| x' - y \|_2}} = s
\end{equation} 

\noindent Since $x' - y = \alpha w $ for some $0 < \alpha \leq 1$, then applying (\ref{technical same norm on ray}) we have  

\begin{equation}\label{technical same Lipschitz constant}
\displaystyle{\frac{\delta \| w \|_1}{ \| x' - y \|_1} = \frac{\delta \| w \|_2}{ \| x' - y \|_2} } = s.
\end{equation}

\noindent We define $T':[x , y] \to [0 , \delta w] \subset [0 , w]$ by 

\begin{equation*}
T'(v) = 
\begin{cases}
\beta \delta w , & \text{if } v  = \beta x' + (1- \beta )y \text{ for some } 0 \leq \beta \leq 1  \\
\delta w , & \text{if } v \in [x , y] \setminus [x' - y] \\
\end{cases}
\end{equation*} 

\noindent Employing (\ref{technical same infinity norm}) and (\ref{technical same Lipschitz constant}) it is not hard to check that $K(T', \| \cdot \|_1) = K(T' , \| \cdot \|_2) = s$ and $\| T' \|_{1 , \infty} = \| T' \|_{2 , \infty} = \delta r$.

By Lemma \ref{retract lemma} applied to $T'$ and norms $\| \cdot \|_1$ and $\| \cdot \|_2$ there exist extensions $T_1$ and $T_2$ of $T'$ from $C$ to $[0, \delta w]$ such that

\begin{equation}
\| T_1 \|_{1 , \infty} = \| T_2 \|_{2 , \infty} = \delta r
\end{equation}

\begin{equation}
K(T_1 , \| \cdot \|_1) = K(T_2 , \| \cdot \|_2) = s
\end{equation} 

\noindent Since $T_1$ and $T_2$ are an extensions of $T'$, then 

\begin{equation}
K(T_1 , \| \cdot \|_2) = s_1 \text{ and } K(T_2, \| \cdot \|_1) = s_2
\end{equation} 

\noindent for some $s_1 , s_2 \geq s$. For each $0 \leq \lambda \leq 1$ we define $T_\lambda = \lambda T_1 + (1 - \lambda) T_2$, it can be shown that 

\begin{equation}
\|T_\lambda \|_{1 , \infty} = \| T_\lambda \|_{2 , \infty} = \delta r
\end{equation}

\noindent and 

\begin{equation}
K(T_\lambda , \| \cdot \|_1) , K(T_\lambda , \| \cdot \|_2) \geq s
\end{equation}

\noindent for each $0 \leq \lambda \leq 1$. We define real continuous functions $f_1 (\lambda) = K(T_\lambda , \| \cdot \|_1)$ and $f_2 (\lambda) = K(T_\lambda , \| \cdot \|_2)$. We note that 
$f_1(0)  = s_2  \geq s  =  f_1(1)$ and $f_2(0)  = s  \leq  s_1  =  f_1(1)$, hence there exists $0 \leq \lambda_0 \leq 1$ such that $f_1(\lambda_0)= f_2(\lambda_0)$. That is, 

\begin{equation}\label{same Lipschitz constant}
K(T_{\lambda_0} , \| \cdot \|_1) = K(T_{\lambda_0} , \| \cdot \|_2) = s_0 \geq s. 
\end{equation}

\noindent We define 

\begin{equation}
T =  \displaystyle{\frac{s}{s_0}} T_{\lambda_0} + \left(1 - \displaystyle{\frac{s}{s_0}}\delta \right)w
\end{equation}

\noindent By (\ref{same Lipschitz constant}) and the fact that $T$ is a translation of $\displaystyle{\frac{s}{s_0}} T_{\lambda_0}$ we have that 

\begin{equation}
K(T , \| \cdot \|_i) = K\left(\displaystyle{\frac{s}{s_0}} T_{\lambda_0} , \| \cdot \|_i \right) = \displaystyle{\frac{s}{s_0}} s_0 = s
\end{equation}

\noindent for $i= 1 , 2$. Since $ 0 \leq \frac{s}{s_0} \leq 1$, then the image of $\displaystyle{\frac{s}{s_0}} T_{\lambda_0}$ is equals to $ \left[ 0 ,  \displaystyle{\frac{s}{s_0}} \delta w \right] \subset [0 , w]$. Hence the image of $T$ is equals to $\left[\left(1 - \displaystyle{\frac{s}{s_0}}\delta \right)w , w\right]$ and by (\ref{technical same infinity norm}) we have that

\begin{equation}
\| T \|_{1, \infty} = \| T \|_{2, \infty} = r.
\end{equation}
\end{proof}

\begin{theorem}\label{Theorem same distance in components}
Let $X$ be a Banach space, $C$ a convex, closed and bounded subset of $X$ with at least two elements, $\varepsilon >0$ and $\overline{\rho \| \cdot \|} , \overline{\rho' \| \cdot \|'} \in (\mathcal{N}(X) \mathcal{N}(\mathbb{R}^2))'$ be two equivalence classes of constructible norms such that $d\left(\overline{\rho \| \cdot \|} , \overline{\rho' \| \cdot \|'}\right) < \varepsilon$. Then for each $\rho_1 \| \cdot \|_1 \in \overline{\rho \| \cdot \|}$ and $\rho' \| \cdot \|' \in \overline{\rho' \| \cdot \|'} $ there exist $\rho_2 \in \mathcal{N}(\mathbb{R}^2)$ and $\| \cdot \|_2 \in \mathcal{N}(X)$ such that $\rho_2 \| \cdot \|_2  = \rho' \| \cdot \|'$,

\begin{equation}
d\left(\overline{\rho_1}, \overline{\rho_2}\right) < \varepsilon
\end{equation}

\noindent and

\begin{equation}
d\left(\overline{\| \cdot \|_1}, \overline{\| \cdot \|_2}\right) < \varepsilon.
\end{equation}

\noindent In other words, if the distance between two constructible norms is less than $\varepsilon$, then representatives can be chosen in such a way that individually their distance is less than $\varepsilon$.
\end{theorem}
\begin{proof}
Let $\varepsilon > 0$, $\rho_1 \| \cdot \|_1 \in \overline{\rho \| \cdot \|}$, $\rho' \| \cdot \|' \in \overline{\rho' \| \cdot \|'}$ and $u \geq l > 0$ sharp such that for each $T \in BLip(C,X)$

\begin{equation}\label{distance for constructible pair with u and l}
l \rho_1 \| T \|_1 \leq \rho' \| T \|' \leq u \rho_1 \| T \|_1.
\end{equation}

\noindent Hence $\displaystyle{\log \frac{u}{l} <  \varepsilon}$. For each $x \in X$ we call the constant function $f_x:C \to X$ defined by $f_x(y) = x$ for each $y \in C$. It is clear that for each constructible norm $\rho \| \cdot \|$ we have that

\begin{equation}
\rho \| f_x \| = \rho (\| f_x \|_\infty , K(f_x , \| \cdot \|)) = \rho(1 , 0) \| x \|.
\end{equation}

\noindent Thus applying (\ref{distance for constructible pair with u and l}) to constant functions we have that 

\begin{equation}
l \rho_1(1 , 0) \| x \|_1 \leq \rho'(1 , 0) \| x \|' \leq u \rho_1 (1, 0) \| x \|_1.
\end{equation}

\noindent for each $x \in X$. By Lemma \ref{properties of rho} it follows that

\begin{equation}\label{distance of base norms for constructible norms}
\begin{split}
d(\| \cdot \|_1 , \| \cdot \|') & = d(\rho_1(1 , 0) \| \cdot \|_1 , \rho'(1, 0) \| \cdot \|') \\
                                 & \leq \log \displaystyle{\frac{\rho_1(1,0) u}{\rho_1(1, 0)l} = \log \frac{u}{l}} \\ 
                                 & < \varepsilon.
\end{split}
\end{equation}

\noindent Note that (\ref{distance of base norms for constructible norms}) is valid for any representatives $\| \cdot \|_1$ and $\| \cdot \|'$ of classes of construtible norms $\overline{\rho_1 \| \cdot \|_1}$ and $\overline{\rho' \| \cdot \|'}$. In fact, by Lemma \ref{Lemma form of constructible norms} we have that these representatives are only scalar multiples. 

By Lemma \ref{Lemma technical indistinguible under two norms} there exists a norm $\| \cdot\|_2 = c \| \cdot \|'$ such that for each $r >0$ and $s \geq 0$ there exists $T \in BLip(C,X)$ with $\| T \|_{1 , \infty} = \| T \|_{2 , \infty}=r$ and  $K(T , \| \cdot \|_1) = K(T , \| \cdot \|_2)=s$. Using (\ref{distance of base norms for constructible norms}) and Lemma \ref{properties of rho} we have that

\begin{equation}
d(\| \cdot \|_1 , \| \cdot \|_2) < \varepsilon.
\end{equation}

\noindent By Lemma \ref{Lemma form of constructible norms} the norm $\rho_2( \cdot , \cdot ) = \rho'(c^{-1} \, \cdot , \cdot )$ satisfy $\rho_2 \| \cdot \|_2 = \rho' \| \cdot \|'$. Since 

\begin{equation}
d(\rho_1 \| \cdot \|_1 , \rho_2 \| \cdot \|_2) = d (\rho_1 \| \cdot \|_1 , \rho' \| \cdot \|') < \varepsilon,
\end{equation}

\noindent then exist sharp $u' \geq l' > 0$ such that $\displaystyle{\frac{u'}{l'} = \frac{u}{l} < e^\varepsilon}$ and

\begin{equation}
l' \rho_1 \| S \|_1 \leq \rho_2 \| S \|_2  \leq u' \rho_1 \| S \|_1.
\end{equation}

\noindent for each $S \in BLip(C, X)$. Let $r > 0$ and $s \geq 0$. Then by Lemma \ref{Lemma technical indistinguible under two norms} there exists $T \in BLip(C,X)$ such that   $\| T \|_{1 , \infty} = \| T \|_{2 , \infty} = r$ and  $K(T , \| \cdot \|_1) = K(T , \| \cdot \|_2)=s$. That is,

\begin{equation}
\begin{split}
l' \rho_1 (r , s)  & = l' \rho_1 \| T \|_1 \\
                   & \leq \rho_2 \| T \|_2  = \rho_2 (r,s)  \\
                   & \leq u' \rho_1 \| T \|_1 \\
                   & = u' \rho_1(r , s).
\end{split}
\end{equation}

\noindent Hence $l '\rho_1(\cdot , \cdot ) \leq \rho_2(\cdot , \cdot ) \leq u' \rho_1(\cdot , \cdot )$ and

\begin{equation}
d(\rho_1 , \rho_2 ) \leq \log \displaystyle{\frac{u'}{l'}} < \varepsilon. 
\end{equation}
\end{proof}

The following corollary is a particular case of the previous theorem, in which an additional condition is added to take the constructible norms sharp one inside the other.

\begin{corollary}\label{Corollary technical c leq u constant projection}
Let $X$ be a Banach space, $C$ a nonempty convex, closed and bounded subset of $X$, $\rho_1 \| \cdot \|_1 , \rho' \| \cdot \|'$ be two constructible norms for $BLip(C,X)$ that satisfy in a sharp way 

\begin{equation}
\| \cdot \|' \leq \| \cdot \|_1 \leq u \| \cdot \|'
\end{equation}
 
\noindent for some $u \geq 1$ with $\log u \leq d(\rho_1 \| \cdot \|_1 , \rho' \| \cdot \|')$. Then the norms $\rho_2 \in \mathcal{N}(\mathbb{R}^2)$ and $\| \cdot \|_2$ whose existence is assured in Theorem \ref{Theorem same distance in components} necessarily fulfill:

\begin{itemize}
\item[(1)] $\| \cdot \|_2 = c \| \cdot \|'$ for some $1 \leq c \leq u$.
\item[(2)] For each $r> 0$ and $s \geq 0$ there exists $T \in BLip(C,X)$ such that

\begin{equation}
\| T \|_{1, \infty} = \| T \|_{2 , \infty} = r
\end{equation}

\noindent and

\begin{equation}
K(T , \| \cdot \|_1) = K(T, \| \cdot \|_2)=s.
\end{equation}
\end{itemize}
\end{corollary}

\begin{proof}
Note that in second paragraph of the proof of Theorem \ref{Theorem same distance in components} the existence of such $c > 0$ is assured and this depends on Lemma \ref{Lemma technical indistinguible under two norms}. While in the proof of such lemma the $c>0$ in question is the one that ensures $\| w' \|_1 = c \| w' \|'$ for any fixed $w' = x - y $ with $x , y \in C$ and $x \neq y$, and $\| \cdot \|_2 = c \| \cdot \|'$. By hypothesis for each $v \in (C-C) \setminus \{ 0 \}$ we have that

\begin{equation}
\| v \|' \leq \| v \|_1 \leq u \| v \|',
\end{equation}

\noindent in particular, $\|w\|' \leq \| w \|_1 \leq u \|w \|' $. Then $\|w\|_1 = c \|w\|'$ implies $1 \leq c \leq u$.

Property $(2)$ is proven explicitly in the first lines of the second paragraph of the proof of Theorem \ref{Theorem same distance in components}. 
\end{proof}

In the following lemma we justify the use norms only on the first cuadrant of $\mathbb{R}^2$.

\begin{lemma}\label{inequality of norms over positive face of R^2}
Let $\rho_1 , \rho_2 \in \mathcal{N}(\mathbb{R}^2)$ and $u  \geq l >0 $ such that for each $r > 0$ and $s \geq 0$

\begin{equation}
l \rho_1( r , s) \leq \rho_2 (r , s) \leq u \rho_1 (r , s).
\end{equation}

\noindent Then there exist $\rho'_1 , \rho'_2 \in \mathcal{N}(\mathbb{R}^2)$ such that $\rho'_1(r , s) = \rho_1 (r , s)$ and $\rho'_2(r , s) = \rho_2 (r,s)$ for each $r > 0$ and $s \geq 0$,  and for each $(a , b ) \in \mathbb{R}^2$

\begin{equation}
l \rho'_1 (a , b) \leq \rho'_2(a , b) \leq u \rho'_1(a , b).
\end{equation}
\end{lemma}

\begin{proof}
Let $B_1$ and $B_2$ be the unitary balls of the norms $\rho_1$ and $\rho_2$, $C'_1$ and $C'_2$ be the fraction of $B_1$ and $B_2$ that are located in the first and third quadrants without the imaginary line. Since $B_1$ and $B_2$ are convex, then $C_1 := conv (C'_1) \subset B_1$ and $C_2 := conv(C'_2) \subset B_2$. It is not difficult to verify that $C_1$ and $C_2$ are like balls $B_1$ and $B_2$, but with flat faces in the second and fourth quadrants. By (\ref{inequality of norms over positive face of R^2}) we have that for each $(a , b)$ in the first and third quadrants

\begin{equation}
l \rho_1 (a , b) \leq \rho_2(a , b) \leq u \rho_1(a , b),
\end{equation}

\noindent that is,

\begin{equation}
u^{-1} C'_1 \subset C'_2 \subset l^{-1} C'_1.
\end{equation}  

\noindent Thus

\begin{equation}
u^{-1} C_1 \subset C_2 \subset l^{-1} C_1.
\end{equation} 

\noindent Finally, the Minkowski norms $\rho'_1$ and $\rho'_2$ associated with $C_1$ and $C_2$ satisfy the required inequality. 
\end{proof}

\section{Completeness of constructible norms}

\begin{theorem}
Let $X$ be a Banach space and $C \subset X$ be a convex, closed and bounded set. Then the family of constructible norms modulus collinearity $(\mathcal{N}(\mathbb{R}^2 )\mathcal{N}(X))'$ for $BLip(C,X)$ is a $d$-complete metric space.
\end{theorem}

\begin{proof}
We will divide this demonstration into two parts, in the first we will choose suitable representatives of a $d$-Cauchy sequence of constructible norms and in the second part with said representatives we will show that in fact the limit of constructible norms is constructible.

First part: Let $\varepsilon> 0$, $\mathfrak{U}$ be a nontrivial ultrafilter over $\mathbb{N}$ and $\left( \overline{\eta_n(\cdot )} \right)$ be a $d$-Cauchy sequence in $(\mathcal{N}(\mathbb{R}^2 )\mathcal{N}(X))'$. Without lost of generality we may assume that for each $n \in \mathbb{N}$ it is fulfilled 

\begin{equation}
d\left(\overline{\eta_n (\cdot )} , \overline{\eta_1(\cdot)}\right) < \varepsilon.
\end{equation}

\noindent We chose $\rho_1 \| \cdot \|_1 \in \overline{\eta_1 ( \cdot )}$ for some fixed $\rho_1 \in \mathcal{N}(\mathbb{R}^2)$ and $\| \cdot \|_1 \in \mathcal{N}(X)$. For each $n \geq 1$ we call $\eta_n(\cdot ) \in \overline{\eta_n(\cdot )}$ the sharpest representative under $\rho_1\|\cdot \|_1$, that is, for each $n \in \mathbb{N}$ there exists sharp $u_n \geq 1$ with

\begin{equation}\label{technical distance eta_n and eta_1}
\eta_n(\cdot ) \leq \rho_1 \| \cdot \|_1 \leq u_n \eta_n(\cdot)
\end{equation} 

\noindent and $d(\eta_n (\cdot ) , \rho_1 \| \cdot \|_1) = \log u_n < \varepsilon$. By Theorem \ref{N is complete} there exists 

\begin{equation}\label{technical eta as d lim of eta_n}
\overline{\eta(\cdot)} := d\text{ -}\lim \overline{\eta_n(\cdot)}  \in \mathcal{ N}'(BLip(C,X)).
\end{equation}

\noindent Using Corollary \ref{Corollary form of limit as U pointwise limit} and (\ref{technical distance eta_n and eta_1}) we have that the function defined for each $T \in BLip(C,X)$ by

\begin{equation}\label{technical eta lim as eta_n lim}
\eta(T) = \displaystyle{\lim_\mathfrak{U}} \, \eta_n(T)
\end{equation}

\noindent is a representative of the $d$-limit $\overline{\eta(\cdot )}$. By Theorem \ref{EquivSeparation} and Lemma \ref{TecCol} all of the possible $\mathcal{N}(x)$ components of $\eta_n(\cdot )$ are collinear to $\phi(\eta_n(\cdot))$, thus for each $n \in \mathbb{N}$ we can chose the representative $\| \cdot \|'_n \in \overline{\phi(\eta_n(\cdot))}$ sharp under $\| \cdot \|_1$ that construct $\eta_n(\cdot )$. That is, for each $n \in \mathbb{N}$ we have that $\eta_n (\cdot ) = \theta_n \| \cdot \|'_n$ for some $\theta_n \in \mathcal{N}(\mathbb{R}^2)$ and there exists $v_n \geq 1$ sharp such that for each $x \in X$

\begin{equation}\label{technical distance of projections}
\| x \|'_n \leq \|x \|_1 \leq v_n \| x \|'_n.
\end{equation}

\noindent By Lemma \ref{Lemma phi is nonexpansive} and (\ref{technical distance eta_n and eta_1}) and (\ref{technical distance of projections}) we have that $0 \leq \log v_n \leq \log u_n < \varepsilon $. Hence $(\overline{\| \cdot \|}'_n)$ is $d$-Cauchy and by Corollary \ref{Corollary form of limit as U pointwise limit} and (\ref{technical distance of projections}) we have that a representative of the $d$-limit of $(\overline{\| \cdot \|}'_n)$ is 

\begin{equation}\label{technical form of limit of norms proof complete contructible}
\| x \|' = \displaystyle{\lim_\mathfrak{U}}\, \| x \|'_n.
\end{equation}

\noindent By Theorem \ref{Theorem same distance in components} and Lemma \ref{TecCol} for each $n \in \mathbb{N}$ we can chose $\rho_n \in \mathcal{N}(\mathbb{R}^2)$ and $\| \cdot \|_n \in \overline{\phi ( \eta_n(\cdot) )} \subset \mathcal{N}(X)$ such that 

\begin{equation}\label{technical representatives rho_n and norm_n for eta_n}
\rho_n \| \cdot \|_n = \eta_n ( \cdot),
\end{equation}

\begin{equation}
d(\rho_n , \rho_1) < \varepsilon
\end{equation}

\noindent and 

\begin{equation}
d(\| \cdot \|_n , \| \cdot \|_1) < \varepsilon.
\end{equation}

\noindent By Corollary \ref{Corollary technical c leq u constant projection} and (\ref{technical distance of projections}) for each $n \in \mathbb{N}$ there exists $1 \leq c_n \leq u_n$ such that 

\begin{equation}\label{technical form of norm n}
\| \cdot \|_n = c_n \| \cdot \|'_n
\end{equation}

\noindent and for each $r >0$, $s\geq 0$ there exists $T \in BLip(C,X)$ such that 

\begin{equation}\label{technical same infinity norm of norm_n and norm_1}
\| T \|_{1, \infty} = \| T \|_{n, \infty} = r
\end{equation}

\noindent and 

\begin{equation}\label{technical same lipschitz value of norm_n and norm_1}
K(T , \| \cdot \|_1) = K(T , \| \cdot \|_n) = s.
\end{equation}

\noindent Note that $0 \leq \log c_n \leq \log v_n \leq \log u_n \leq \varepsilon$, thus $c_n$ is bounded, then exists $c = \displaystyle{\lim_\mathfrak{U}}\, c_n$. Hence by (\ref{technical form of limit of norms proof complete contructible}) for each $x \in X$

\begin{equation}\label{technical lim norm as U lim o n norms}
\begin{split}
\| x \| & := \displaystyle{\lim_\mathfrak{U}} \, \|x \|_n \\ 
        & = \displaystyle{\lim_\mathfrak{U}} \, c_n \| x \|'_n \\
        & = c \| x \|'
\end{split}
\end{equation}

\noindent and by Lemma \ref{Lemma infinity norm as limit of infinity norms} for each $T \in BLip(C, X)$

\begin{equation}\label{technical infinity norm as limit of infinity norms}
\| T \|_\infty = \displaystyle{\lim_\mathfrak{U}} \, \| T \|_{n , \infty}.
\end{equation}

\noindent By (\ref{technical distance eta_n and eta_1}) and (\ref{technical representatives rho_n and norm_n for eta_n}) for each $n \in \mathbb{N}$ and $T \in BLip(C,X)$ we have that

\begin{equation}
\rho_n\| T \|_n \leq \rho_1 \| T \|_1 \leq u_n \rho_n \| T \|_n. 
\end{equation}

\noindent By (\ref{technical same infinity norm of norm_n and norm_1}) and (\ref{technical same lipschitz value of norm_n and norm_1}) 

\begin{equation}
\rho_n(r,s) \leq \rho_1 (r,s ) \leq u_n \rho_n (r,s)
\end{equation}

\noindent for any $n \in \mathbb{N}$, $r> 0$ and $s\geq 0$. Finally using Lemma \ref{inequality of norms over positive face of R^2} we may assume without lost of generality 

\begin{equation}\label{technical bound for rho_n}
\rho_n(a,b) \leq \rho_1 (a,b ) \leq u_n \rho_n (a,b)
\end{equation}

\noindent for each $(a, b) \in \mathbb{R}^2$.

Second part: By Theorem \ref{Theorem construct Intrin Charac} it is enough to prove that $\eta(S) = \eta(T)$ for each $\phi(\eta(\cdot ))$-indistinguishable $S, T \in BLip(C , X)$. By (\ref{technical representatives rho_n and norm_n for eta_n}) we have that 

\begin{equation}\label{technical projection of eta_n}
\begin{split}
\phi(\eta_n(x)) & = \eta_n(f_x) \\
                & = \rho_n \| f_x\|_n \\
                & = \rho_n(1,0) \|x\|_n
\end{split}
\end{equation}

\noindent Using (\ref{technical eta lim as eta_n lim}) and (\ref{technical projection of eta_n}) we note that

\begin{equation}\label{technical form of eta projection}
\phi(\eta(x))  = \displaystyle{\lim_\mathfrak{U}} \, \rho_n(1 , 0) \|x \|_n
\end{equation}

\noindent From (\ref{technical bound for rho_n}) we infer that

\begin{equation}\label{technical alpha is the value of U lim rho_n(1,0)}
0 < \alpha = \displaystyle{\lim_\mathfrak{U}} \, \rho_n(1 , 0) \leq \rho_1(1,0)
\end{equation}

\noindent Hence for each $T \in BLip(C,X)$ by (\ref{technical lim norm as U lim o n norms}), (\ref{technical form of eta projection}) and (\ref{technical alpha is the value of U lim rho_n(1,0)}) we have

\begin{equation}\label{technical infinity norm of projection}
\begin{split}
\phi(\eta(T))_\infty & = \displaystyle{\sup_{x \in C}} \, \phi(\eta (Tx) ) \\
                     & = \displaystyle{\sup_{x \in C}} \, \displaystyle{\lim_\mathfrak{U}} \, \rho_n(1 , 0) \|T x \|_n \\
                     & = \displaystyle{\sup_{x \in C}} \, \alpha  \| Tx \| \\
                     & = \alpha \| T \|_\infty.
\end{split}                  
\end{equation}

\noindent and by Remark \ref{Remark 1}, Lemma \ref{Lemma phi is nonexpansive}, (\ref{technical eta as d lim of eta_n}), (\ref{technical projection of eta_n}), (\ref{technical form of norm n}) and (\ref{technical lim norm as U lim o n norms}) we have 

\begin{equation}\label{technical Lipschitz constant of projection}
\begin{split}
K(T , \phi(\eta(\cdot ))) & = \displaystyle{\lim_\mathfrak{U}} \, K(T , \phi(\eta_n(\cdot))) \\
                          & = \displaystyle{\lim_\mathfrak{U}} \, K(T , \rho_n(0 , 1) \| \cdot \|_n) \\
                          & = \displaystyle{\lim_\mathfrak{U}} \, K(T, \| \cdot \|'_n) \\
                          & = K(T , \| \cdot \|') = K(T , c \| \cdot \|')\\
                          & = K(T , \| \cdot \|).
\end{split}
\end{equation}

Now we will prove that pointwise $\eta(T) = \displaystyle{\lim_\mathfrak{U}} \, \rho_n \| T \|_n = \displaystyle{\lim_\mathfrak{U}} \, \rho_n \| T \|$ for each $T \in BLip(C,X)$. By (\ref{technical eta lim as eta_n lim}), (\ref{technical infinity norm as limit of infinity norms}), Remark \ref{Remark 1} and the fact that $(\overline{\| \cdot \|}_n)$ is $d$-Cauchy, we have that for each $\delta > 0$ and $T \in BLip(C,X)$ there exist $I_1 , I_2, I_3 \in \mathfrak{U}$ such that

\begin{equation}\label{technical bound pointwise distance eta and eta_n}
| \rho_n \| T \|_n - \eta(T)| < \delta
\end{equation}

\noindent for each $n \in I_1$,

\begin{equation}\label{technical bound infinity norm of norm and norm_n}
| \| T \|_{n , \infty} - \| T \|_\infty | < \delta
\end{equation}

\noindent for each $n \in I_2$ and

\begin{equation}\label{technical bound Lipschitz values}
| K(T , \| \cdot \|_n) - K(T , \| \cdot \|) |< \delta
\end{equation}

\noindent for each $n \in I_3$. Thus (\ref{technical bound pointwise distance eta and eta_n}), (\ref{technical bound infinity norm of norm and norm_n}) and (\ref{technical bound Lipschitz values}) are valid for every $n \in I = I_1 \cap I_2 \cap I_3 \in \mathfrak{U}$. By (\ref{technical bound for rho_n}) for each $T \in BLip(C,X)$ and $n \in \mathbb{N}$ we have that

\begin{equation}\label{technical distance of rho norms limits}
\begin{split}
\left| \rho_n \| T \|_n -  \rho_n \| T \| \right|  & = \left| \rho_n (\| T \|_{n , \infty} , K(T , \| \cdot \|_n)) - \rho_n (\| T \|_\infty , K(T , \| \cdot \|) \right| \\
                             & \leq  \rho_n (\| T \|_{n , \infty} - \| T \|_\infty ,  K(T , \| \cdot \|_n) - K(T , \| \cdot \|) ) \\
                             & \leq \rho_1 (\| T \|_{n , \infty} - \| T \|_\infty ,  K(T , \| \cdot \|_n) - K(T , \| \cdot \|)  ).
\end{split}
\end{equation}

\noindent Since (\ref{technical distance of rho norms limits}) is valid for every $n \in  \mathbb{N}$, then by (\ref{technical bound pointwise distance eta and eta_n}), (\ref{technical bound infinity norm of norm and norm_n}) and (\ref{technical bound Lipschitz values}) we have that 

\begin{equation}
\displaystyle{\lim_{I , \mathfrak{U}}} \, \left| \rho_n \| T \|_n -  \rho_n \| T \| \right| =0 
\end{equation}

\noindent for every $T \in BLip(C,X)$. Then

\begin{equation}\label{technical simplified form of eta as limit}
\eta(T) = \displaystyle{\lim_\mathfrak{U}} \, \rho_n \| T \|_n  = \displaystyle{\lim_\mathfrak{U}} \,  \rho_n \| T \|.
\end{equation}

We will finish by proving that $\eta(\cdot )$ is constructible. Let $\phi(\eta(\cdot ))$-indistinguishable $S,T \in BLip(C,X)$. That is,

\begin{equation}
\phi(\eta(S))_\infty = \phi (\eta(T))_\infty
\end{equation}

\noindent and 

\begin{equation}
K(S , \phi (\eta(\cdot))) = K(T , \phi (\eta(\cdot))).
\end{equation}

\noindent By (\ref{technical infinity norm of projection}) and (\ref{technical Lipschitz constant of projection}) we have that

\begin{equation}\label{same infinity value}
\| S \|_\infty = \| T \|_\infty
\end{equation}

\noindent and 

\begin{equation}\label{same lipschitzian value}
K(S , \| \cdot \|) = K(T , \| \cdot \|).
\end{equation}

\noindent Then using (\ref{technical simplified form of eta as limit}), (\ref{same infinity value}) and (\ref{same lipschitzian value}) 

\begin{equation}
\begin{split}
\eta(S) & = \displaystyle{\lim_\mathfrak{U}} \, \rho_n \| s \| \\
        & = \displaystyle{\lim_\mathfrak{U}} \, \rho_n(\| S \|_\infty , K(S , \| \cdot \|)) \\
        & = \displaystyle{\lim_\mathfrak{U}} \, \rho_n(\| T \|_\infty , K(T , \| \cdot \|)) \\
        & = \displaystyle{\lim_\mathfrak{U}} \, \rho_n \| s \| \\
        & = \eta(T).
\end{split}
\end{equation}

\noindent Thus, by Theorem \ref{Theorem construct Intrin Charac} we conclude that $\eta(\cdot)$ is constructible, that is, the $d$-limit of constructible norms is constructible.
\end{proof}

\end{document}